\documentclass[12pt, a4paper]{amsart}
\usepackage[english]{babel}
\usepackage[utf8]{inputenc}
\usepackage[format=plain, font=footnotesize]{caption}
\usepackage{enumitem}
\setlist[enumerate]{leftmargin=7mm, label=\alph*)}
\usepackage[sort]{cite}
\usepackage[a4paper,asymmetric]{geometry}
\usepackage{amsthm,amsmath,amssymb,amscd,amsfonts}
\usepackage{hyperref}
\usepackage{algorithm2e}
\usepackage{xcolor}
\hypersetup{
     colorlinks = true,
     linkcolor = black,
     anchorcolor = blue,
     citecolor = red,
     filecolor = blue,
     urlcolor = blue
     }
\usepackage{url}
\usepackage{graphicx}
\usepackage{mathtools}
\usepackage{verbatim}
\usepackage{listings}
\definecolor{codedarkgreen}{RGB}{51, 133, 4}
\definecolor{codemaroon}{RGB}{133, 5, 63}
\definecolor{codeteal}{RGB}{0, 128, 96}
\lstdefinelanguage{Macaulay2}{
basicstyle=\small\ttfamily,
alsoletter=",
classoffset=1,
keywords={matrix,minors,gb,transpose},
keywordstyle={\color{blue}},
classoffset=2,
morekeywords={from, to, list},
keywordstyle={\color{codemaroon}},
classoffset=3,
morekeywords={QQ},
keywordstyle={\color{codedarkgreen}},
classoffset=4,
morekeywords={MonomialOrder},
keywordstyle={\color{codeteal}},
xleftmargin=1.5cm,
xrightmargin=1em,
columns=fullflexible,
keepspaces=true,
stepnumber=1,
numbers=none,
captionpos=b,
showspaces=false,
frame=none
}

\newtheorem{theorem}{Theorem}[section] 
\newtheorem{proposition}[theorem]{Proposition}
\newtheorem{lemma}[theorem]{Lemma}

\theoremstyle{definition}

\newtheorem{remark}[theorem]{Remark}

\theoremstyle{remark}
\newtheorem*{example}{Example}

\newcommand{\bfx}{\mathbf x}

\newcommand{\bfp}{\mathbf p}
\newcommand{\bfq}{\mathbf q}

\title[A short proof for the parameter continuation theorem]{A short proof for the \\ parameter continuation theorem}

\author[V.~Borovik]{Viktoriia Borovik}
\email{vborovik@uni-osnabrueck.de}

\author[P.~Breiding]{Paul Breiding}
\email{pbreiding@uni-osnabrueck.de}

\thanks{University Osnabr\"uck,
Fachbereich Mathematik/Informatik
Albrechtstr.\ 28a,
49076 Osnabrück, Germany. Both authors are supported by the Deutsche Forschungsgemeinschaft (DFG) -- Projektnummer 445466444}

\begin{document}
\maketitle 

\begin{abstract}
The Parameter Continuation Theorem is the theoretical foundation for polynomial homotopy continuation, which is one of the main tools in computational algebraic geometry. In this note, we give a short proof using Gr\"obner bases. Our approach gives a method for computing discriminants.
\end{abstract}
\medskip

\section{Introduction}
A central task in many applications is solving a system of polynomial equations. One approach to solving such systems is \emph{polynomial homotopy continuation}.  To explain the basic idea we consider the polynomial ring 
$$\mathbb C[\bfx,\bfp]:=\mathbb C[x_1,\ldots,x_n,p_1,\ldots,p_k].$$ 
We interpret $\bfx$ as variables and $\bfp$ as {parameters}.

Let $f_1(\bfx;\bfp),\ldots,f_n(\bfx;\bfp)\in \mathbb C[\bfx,\bfp]$. We call the image of the polynomial map
\begin{equation}\label{def_family}\mathbb C^k \mapsto \mathbb C[\bfx]^{\times n},\quad \bfp\mapsto F(\bfx;\bfp) =\begin{bmatrix}
     \ f_1(\bfx;\bfp)\ \\
     \vdots\\
     \ f_n(\bfx;\bfp)\ 
\end{bmatrix}
\end{equation}
a \emph{family} of polynomial systems. That is, a family $\mathcal F=\{F(\bfx;\bfp) \mid \bfp\in\mathbb C^k\}$ consists of~$n$ polynomials in $n$ variables $\bfx$ with $k$ parameters $\bfp$.

Let $F(\bfx)$ be a system of $n$ polynomials in $n$ variables.
The idea in polynomial homotopy continuation is to find a family $\mathcal F$ and parameters $\bfq_1,\bfq_2\in\mathbb C^k$ with the following properties: $F(\bfx)=F(\bfx;\bfq_1)$ and $G(\bfx)=F(\bfx;\bfq_2)$ is another system whose solutions can be computed or are known. One then defines the \emph{parameter homotopy}
$H(\bfx,t):=F(\bfx; (1-t)\bfq_1 + t\bfq_2)$
and \emph{tracks} the zeros of $H(\bfx,t)$ from $t=1$ to $t=0$. This means, given $\bfx\in\mathbb C^n$ with $G(\bfx)=0$, we use numerical algorithms to solve the \emph{Davidenko ODE} 
$
\big(\tfrac{\mathrm d}{\mathrm d \bfx} H(\bfx,t)\big)\ \dot \bfx + \tfrac{\mathrm d}{\mathrm d t} H(\bfx,t) = 0$ \cite{Davidenko, Davidenko_full} 
for the initial value $\bfx(0)= \bfx$.
In this setting, $G(\bfx)$ is called the \emph{start system} and $F(\bfx)$ is called the \emph{target system}. For more on the theory of polynomial homotopy continuation we refer to the textbook of Sommese and Wampler \cite{Sommese:Wampler:2005} or the overview article \cite{NNLA}. 

The \emph{Parameter Continuation Theorem} by Morgan and Sommese~\cite{MS1989} is the theoretical foundation of polynomial homotopy continuation. It implies that the initial value problem above is well posed for almost all parameters. Recall that a zero $\bfx$ of $F(\bfx;\bfq)$ is called \emph{regular} if the Jacobian determinant $\det\big(\tfrac{\partial f_i}{\partial x_j}\big)_{1\leq i,j\leq n}$ at $(\bfx,\bfq)$ does not vanish.
\begin{theorem}[The Parameter Continuation Theorem]\label{theorem_number_zeros}
Let $\mathcal F$ be a family of polynomial systems that consists of systems of $n$ polynomials $F(\bfx;\bfp)$ in $n$ variables~$\bfx$ depending on $k$ parameters $\bfp$. For $\bfq\in\mathbb C^k$ denote 
$$N(\bfq):=\#\{ \bfx\in\mathbb C^n \mid \bfx\text{ is a regular zero of } F(\bfx;\bfq) \}.$$
Let $N:=\sup_{\bfq\in\mathbb C^k} N(\bfq)$. Then, $N<\infty$ and there exists a proper algebraic subvariety~$\Delta \subsetneq \mathbb C^k$, called a {discriminant}, such that 
$N(\bfq)=N$ for all $\bfq\not\in\Delta$.
\end{theorem}
\begin{example} We consider two examples.
\begin{enumerate}
\item 
The space $\mathcal F = \{ax^2+bx+c\mid a,b,c\in\mathbb C\}$ of univariate quadratic polynomials is a family. Here, there are 3 parameters $\bfp=(a,b,c)$ and we have $N=2$.
\item
Consider the family of polynomials with one parameter $\bfp=a$ defined by
$$F(\bfx;\bfp) = \begin{bmatrix}
     \ (x_1-a) \cdot (x_1 - 1)\phantom{^2} \cdot \gamma(\bfx)\ \\
     \ (x_2-3) \cdot (x_2-4)^2\cdot \gamma(\bfx)\ 
\end{bmatrix},$$
where $\gamma(\bfx)=x_1^2+x_2^2-1$. In this case, $N=2$. Moreover, for all $\bfq\in\mathbb C$ the zero set of $F(\bfx;\bfq)$ contains a curve and at least one singular point. 
\end{enumerate} 
\end{example}

A proof for the Parameter Continuation Theorem can also be found in the textbook \cite{Sommese:Wampler:2005}. The proofs in \cite{MS1989, Sommese:Wampler:2005} rely on the theory of holomorphic vector bundles. In this short note we give an alternative proof using Gr\"obner bases. Our main contribution is derived from Lemma \ref{GB_lemma}, where we show that a discriminant is a lower dimensional subvariety of $\mathbb{C}^k$. This implies path-connectedness of $\mathbb{C}^k\setminus \Delta$. Showing the existence of $N$ homotopy paths of roots can serve as a proof for generic root count theorems, such as the Fundamental Theorem of Algebra (see \cite{BCSS, FTA}) and the BKK Theorem (see \cite{Rojas2002WhyPM}).

Moreover, our proof gives rise to a method for computing discriminants explicitly. We understand a discriminant as a hypersurface $\Delta \subsetneq \mathbb{C}^k$ \eqref{disc} containing all `bad' parameters~$\mathbf{q}$ such that the system $F(\mathbf{x};\mathbf{q})$ has a nonregular zero. The Zariski closure of all such parameters is a variety defined by the elimination ideal 
$$\langle f_1(\bfx;\bfp),\ldots,f_n(\bfx;\bfp), \det\big(\tfrac{\partial f_i}{\partial x_j}\big)_{1\leq i,j\leq n}\rangle \cap \mathbb{C}[\mathbf{p}].$$
This is a generalization of the discriminant $\Delta(f)$ of a univariate polynomial $f$ which vanishes whenever $f$ has a multiple root. On more definitions of discriminants we refer to \cite[Chapters 1,9,12,13]{Gelfand1994DiscriminantsRA}.  A special case of the discriminant, when~$F$ contains~$r$ nonlinear polynomials with the common zero locus $V$ of dimension $n-r$ and $n-r$ linear equations, such that the parameters $\mathbf{p}$ are the coefficients of the linear equations, is called \emph{Hurwitz form of $V$} and was studied in \cite{STURMFELS2017186, 10.1145/3653002.3653003}.
\medskip

\section{Gr\"obner bases, Saturation and Parameterized Ideals}
For the proof of the Parameter Continuation Theorem we use Gr\"obner bases of ideals with parameters. 
Let $f_1(\bfx;\bfp),\ldots,f_n(\bfx;\bfp)$ be polynomials as in (\ref{def_family}) and denote by
$$h(\bfx;\bfp):=\det\big(\tfrac{\partial f_i}{\partial x_j}\big)$$ the Jacobian determinant. We consider two ideals 
$$I:= \langle f_1,\ldots,f_n\rangle\quad\text{and}\quad J:=\big\langle h\rangle.$$
The \emph{saturation} of $I$ by $J$ is the ideal
$$I:J^\infty := \{f \in \mathbb C[\bfx,\bfp]\mid \exists \ell>0: f\cdot h^\ell\in I\}.$$
Saturation corresponds to removing components on the level of varieties: we have  
\begin{equation}\label{saturation_as_variety}\mathbf{V}(I:J^\infty) = \overline{\mathbf{V}(I)\setminus \mathbf{V}(J)};
\end{equation}
see, e.g., \cite[Chapter 4 \S 4, Corollary 11]{CLO15}. 

Let now $\bfq\in\mathbb C^k$ be fixed and define the surjective ring homomorphism
$$\phi_{\bfq}: \mathbb C[\bfx,\bfp]\to \mathbb C[\bfx],\quad f(\bfx;\bfp)\mapsto f(\bfx;\bfq).$$
We consider the two ideals
$$I_{\bfq} := \phi_{\bfq}(I)\quad \text{and}\quad J_{\bfq} := \phi_{\bfq}(J).$$
The Implicit Function Theorem implies that a regular zero of $F(\bfx;\bfq)$ is an isolated point in $\mathbf{V}(I_\bfq)$. Suppose $N(\bfq)=\infty$. Then,  $\mathbf{V}(I_\bfq:J_\bfq^\infty)=\overline{\mathbf{V}(I_\bfq) \setminus \mathbf{V}(J_\bfq)}$ would be of positive dimension contradicting that regular zeros are isolated. Hence, $N(\bfq)<\infty$. A finite union of points is Zariski closed, so $\mathbf{V}(I_{\bfq}:J_{\bfq}^\infty)$ is the set of regular zeros of~$F(\bfx;\bfq)$. This implies 
\begin{equation}\label{key_idea}
N(\bfq) = \#  \mathbf{V}(I_\bfq:J_\bfq^\infty).
\end{equation}

The basic idea for our proof of the Parameter Continuation Theorem is to show that $N(\bfq)$ is the degree of the projection $\mathbf{V}(I:J^\infty)\to \mathbb C^k$, which maps~$(\bfx,\bfq)$ to the parameter $\bfq$. We show that on the level of ideals we have $\phi_{\bfq}(I:J^\infty) = I_{\bfq} : J_{\bfq}^\infty$. One inclusion is straight-forward: suppose that $f\in I:J^\infty$ with $f\cdot h^\ell\in I$. Then, $\phi_{\bfq}(f)^\ell \cdot \phi_{\bfq}(h) = \phi_{\bfq}(f^\ell\cdot h)\in I_{\bfq}$. The other inclusion, however, does not hold in general. Our main contribution is Lemma~\ref{GB_lemma}, where we show that it holds for a general $\bfq$.  

In the following, we consider the polynomial ring $\mathbb C[\bfx,\bfp]$ with the \emph{lex order}
\begin{equation}\label{lex}
x_1>\cdots>x_n>p_1>\cdots>p_k.
\end{equation}

We recall two propositions related to elimination and saturation of ideals with parameters.
\begin{proposition}\label{sat_thm}
Let $I\subset \mathbb C[\bfx,\bfp]$ be an ideal and $J=\langle h\rangle$ be a principal ideal. Let $y$ be an additional variable and
$K:=\langle 1-y\cdot h\rangle.$
Then,
$$I:J^\infty =(I+K) \cap \mathbb C[\bfx, \bfp].$$
Furthermore, if we augment the lex order~(\ref{lex}) by letting $y$ be the largest variable and let $G$ be a Gröbner basis of $I+K$ relative to this order, then $G \cap \mathbb C[\bfx, \bfp]$ is a Gröbner basis of $I:J^\infty$.
\end{proposition}
\begin{proof}
See \cite[Chapter 4 \S 4, Theorem 14]{CLO15}.
\end{proof}
\begin{proposition}\label{prop_evaluation}
Consider an ideal $L\subset \mathbb C[\bfx,\bfp]$ and let $G = \{g_1,\ldots,g_s\}$ be a Gr\"obner basis for $L$ relative to lex order in (\ref{lex}). For $1\leq i\leq s$ with $g_i \not\in\mathbb C[\bfp]$, write $g_i$ in the form $g_i = c_i(\bfp)\bfx^{\alpha_i} + h_i$, where all terms of $h_i$ are strictly smaller than $\bfx^{\alpha_i}$. Let $\bfq\in \mathbf{V}(L \cap \mathbb C[\bfp] )\subseteq \mathbb C^{k}$, such that $c_i(\bfq)\neq0$ for all $g_i\not\in\mathbb C[\bfp]$. Then,
$$\phi_{\bfq}(G)=\{\phi_{\bfq}(g_i) \mid g_i\not\in\mathbb C[\bfp]\}$$ 
is a Gröbner basis for the ideal $\phi_{\bfq}(L)\subset \mathbb C[\bfx]$.
\end{proposition}
\begin{proof}
See \cite[Chapter 4 \S 7, Theorem 2]{CLO15}.
\end{proof}
\begin{remark}
Weispfenning \cite{WEISPFENNING19921} proved the existence of a \emph{Comprehensive Gröbner Basis}  (CGB). In the notation of Proposition \ref{prop_evaluation} this is a Gr\"obner basis $G$ of $L$, such that $\phi_{\bfq}(G)$ is a Gr\"obner basis for $\phi_{\bfq}(L)$ for \emph{all} $\bfq \in \mathbb C^k$. 
We also refer the reader to the following works of Weispfenning, Montes, Kapur and others \cite{WEISPFENNING2003669, montes1999basic, montes2002new, KAPUR2013124}, which considerably improved the construction and optimized the algorithm for computing a CGB. Nevertheless, for our purposes, Proposition \ref{prop_evaluation} is enough. In addition, the condition of non-vanishing leading coefficients is essential for us.
\end{remark}

The next lemma is our main contribution.
\begin{lemma}\label{GB_lemma}
Let $I\subset \mathbb C[\bfx,\bfp]$ be an ideal and $J=\langle h\rangle$ be a principal ideal, such that $(I:J^\infty) \cap \mathbb C[\bfp] = \{0\}$. Let $G=\{g_1,\ldots,g_s\}$ be a Gr\"obner basis of $I:J^\infty$ relative to the lex order~(\ref{lex}). There is a proper subvariety $\Delta\subsetneq \mathbb C^k$ such that for all $\bfq\not\in\Delta$ the set $\{\phi_{\bfq}(g_1),\ldots,\phi_{\bfq}(g_s)\}$ is a Gr\"obner basis for $\phi_\bfq(I):\phi_\bfq(J)^\infty$ and none of the leading terms of $g_1,\ldots,g_s$ vanish when evaluated at $\bfq$.

In particular, $\phi_\bfq(I:J^\infty) = \phi_\bfq(I):\phi_\bfq(J)^\infty$ for all $\bfq \not\in\Delta$.
\end{lemma}
\begin{proof}
Let $y$ be an additional variable and, as in Proposition \ref{sat_thm}, denote 
$$K:= \langle 1 - y\cdot h\rangle.$$
By Proposition \ref{sat_thm}, we have $I:J^\infty = (I+K)\cap \mathbb C[\bfx, \bfp]$.
Since $(I:J^\infty) \cap \mathbb C[\bfp] = \{0\}$, we therefore have 
\begin{equation}\label{empty_inter}
     (I + K) \cap  \mathbb C[\bfp] = \{0\}.
\end{equation}
In particular, $\mathbf{V}((I + K) \cap  \mathbb C[\bfp]) = \mathbb C^k$ and we may therefore apply Proposition~\ref{prop_evaluation} to $I + K$ without putting any restrictions on $\bfq$.

As in Proposition \ref{prop_evaluation}, we augment the lex order~(\ref{lex}) by letting $y$ be the largest variable.
Let $\overline{G}:=\{g_1,\ldots,g_r\}$ be a Gröbner basis of $I+K$ relative to this order. It follows from (\ref{empty_inter}) that we have $g_1,\ldots,g_r \notin \mathbb C[\bfp]$.
We write each $g_i$ in the form~$g_i = c_i(\bfp)y^{\beta}\bfx^{\alpha_i} + h_i$, where all terms of $h_i$ are strictly smaller than~$y^{\beta}\bfx^{\alpha_i}$, and define the hypersurface 
\begin{equation}\label{disc}
\Delta:=\{\bfq\in \mathbb C^k\mid c_1(\bfq) \cdots c_r(\bfq)=0\}.
\end{equation} 

In the following, let $\bfq \in \mathbb C^k \setminus \Delta$.
By Proposition \ref{prop_evaluation}, $\phi_{\bfq}(\overline{G})=\{\phi_{\bfq}(g_1) ,\ldots,\phi_{\bfq}(g_r) \}$ is a Gr\"obner basis for 
$$\phi_{\bfq}(I + K) = \phi_{\bfq}(I) + \phi_{\bfq}(K) = \phi_{\bfq}(I) + (1-y\cdot \phi_{\bfq}(h)).$$
Without restriction, the first $s\leq r$ elements in $\overline{G}$ are those that do not depend on~$y$. We denote $G:=\{g_1,\ldots,g_s\} =  \overline{G}\cap \mathbb C[\mathbf{x}, \mathbf{p}]$.
It follows from Proposition \ref{sat_thm} that $G$ is a Gröbner basis of~$I:J^\infty$. Because $\bfq\not\in \Delta$, none of the leading terms in~$\overline{G}$ when evaluated at $\bfq$ vanish. Consequently, $$\phi_{\bfq}(G) \cap \mathbb C[\mathbf{x}] = \phi_{\bfq}(\overline{G}) \cap \mathbb C[\mathbf{x}].$$ Therefore, $\phi_\bfq(G)=\{\phi_{\bfq}(g_1),\ldots,\phi_{\bfq}(g_s)\}$ is a Gröbner basis of $\phi_\bfq(I):\phi_\bfq(J)^\infty$ by Proposition \ref{sat_thm}.
\end{proof}

\begin{example}
We illustrate Lemma \ref{GB_lemma} using the two examples from the introduction.
\begin{enumerate}
\item For $\mathcal F = \{ax^2+bx+c\mid a,b,c\in\mathbb C\}$ we have $I=\langle ax^2+bx+c\rangle$ and $J=\langle 2ax+b\rangle$. We first compute a Gr\"obner basis for $I:J^\infty$ using \texttt{Macaulay2} \cite{M2}:
\begin{lstlisting}[language=Macaulay2]
R = QQ[x, a, b, c, MonomialOrder => Lex];
f = a * x^2 + b * x + c; h = 2 * a * x + b;
I = ideal {f}; 
J = ideal {h};
S = saturate(I, J); 
G = gens gb S
\end{lstlisting}
This yields $G=\{ax^2+bx+c\}$. Now we consider two sets of parameters. First, $\bfq_1=(1,3,2)$ and then $\bfq_2=(1,-2,1)$.
\begin{lstlisting}[language=Macaulay2]
Iq = sub(I, {a=>1, b=>3, c=>2});
Jq = sub(J, {a=>1, b=>3, c=>2});
Sq = saturate(Iq, Jq); Gq = gens gb Sq
\end{lstlisting}
This gives us the Gr\"obner basis $\{\phi_{\bfq_1}(ax^2+bx+c)\} = \{x^2+3x+2\}$ for $\phi_{\bfq_1}(I)\colon \phi_{\bfq_1}(J)^{\infty}$. 
On the other hand, $\phi_{\bfq_2}(ax^2+bx+c) = x^2-2x+1 = (x-1)^2$, so that in this case, $\phi_{\bfq_2}(I)\colon \phi_{\bfq_2}(J)^{\infty} = \langle 1\rangle$.
\item For the second example we also compute Gr\"obner basis for $I:J^\infty$ using \texttt{Macaulay2} \cite{M2}:
\begin{lstlisting}[language=Macaulay2]
R = QQ[x1, x2, a, MonomialOrder => Lex];
gamma = x1^2 + x2^2 - 1;
f1 = (x1 - a) * (x1 - 1) * gamma;
f2 = (x2 - 3) * (x2 - 4)^2 * gamma;
h = diff(x1, f1) * diff(x2, f2) - diff(x1, f2) * diff(x2, f1);
I = ideal {f1, f2}; 
J = ideal {h};
S = saturate(I, J); 
G = gens gb S
\end{lstlisting}
This yields the Gr\"obner basis 
$G = \{x_2-3,\ x_1^2-ax_1-x_1+a\}$.
We compute the saturation for $\bfq_1=1$ and $\bfq_2=2$. In the first case:
\begin{lstlisting}[language=Macaulay2]
Iq = sub(I, {a=>1});
Jq = sub(J, {a=>1});
Sq = saturate(Iq, Jq)
\end{lstlisting}
gives the ideal $\phi_{\bfq_1}(I)\colon \phi_{\bfq_1}(J)^{\infty} = \langle 1 \rangle$. This is because for $a=1$ the two regular zeros $(a,3)$ and $(1,3)$ come together to  form a singular zero. On the other hand,~$\phi_{\bfq_2}(I)\colon \phi_{\bfq_2}(J)^{\infty}$ has Gr\"obner basis $\phi_{\bfq_2}(G) = \{x_2-3,\ x_1^2-3x_1+2\}$.
\end{enumerate}
\end{example}

We now prove the Parameter Continuation Theorem.
\begin{proof}[Proof of Theorem \ref{theorem_number_zeros}]
If $N=0$, then no system in $\mathcal F$ has regular zeros. In this case, the statement is true. 

We now assume $N>0$. By (\ref{saturation_as_variety}), 
we have $\mathbf{V}(I:J^\infty) = \overline{\mathbf{V}(I)\setminus \mathbf{V}(J)}$. Since
\begin{align*}
 {\bf V}(I) &= \{({\bf x},{\bf q}) : {\bf x} \text{ is a zero of } F({\bf x};{\bf q}) \}  \; \text{ and}\\
{\bf V}(I ) \cap {\bf V}(J ) &= \{({\bf x},{\bf q}) : {\bf x} \text{ is not a regular zero of } F({\bf x};{\bf q})\}, 
\end{align*}
$\mathbf{V}(I:J^\infty)$ is the closure of all $(\bfx,\bfq)\in\mathbb C^n\times \mathbb C^k$ such that~$\bfx$ is a regular zero of~$F(\bfx;\bfq)$.
Since $N>0$, we have $\mathbf{V}(I:J^\infty)\neq \emptyset$. Let $(\bfx,\bfq)\in \mathbf{V}(I:J^\infty)$ such that ${\bf x}$ is a regular zero of $F({\bf x};{\bf q})$. The~Implicit Function Theorem implies that there is a Euclidean open neighbourhood~$U$ of $\bfq$ such that~$F(\bfx;\bfq)$ has regular zeros for all~$\bfq\in U$. Consequently, $(I:J^\infty) \cap \mathbb C[\bfp] = \{0\}$, so we can apply Lemma~\ref{GB_lemma}.

As before, we denote $I_\bfq=\phi_\bfq(I)$ and $J_\bfq=\phi_\bfq(J)$. By Lemma \ref{GB_lemma}, we have $\phi_{\bfq}(I:J^\infty) = I_{\bfq} : J_{\bfq}^\infty$ for general $\bfq$. This implies that $\mathbf{V}(I_\bfq:J_\bfq^\infty)$ is the fiber of the projection  $\mathbf{V}(I:J^\infty)\to \mathbb C^k,\, (\bfx,\bfq)\mapsto \bfq$ for general $\bfq$. Therefore,  $\# \mathbf{V}(I_\bfq:J_\bfq^\infty)$ is the degree of this projection, hence constant and maximal for general $\bfq$. The statement now follows from $N(\bfq)=\# \mathbf{V}(I_\bfq:J_\bfq^\infty)$, which we showed in (\ref{key_idea}).
\end{proof}
\begin{remark}
The following yields a more explicit proof for Theorem \ref{theorem_number_zeros}, not working with degrees of projections. 

The idea is to show that for general $\bfq$ the Gr\"obner bases of $I_{\bfq}:J_{\bfq}^\infty$ all have the same number of \emph{standard monomials}. Recall that one calls a monomial $\bfx^\alpha$ a {standard monomial} of an ideal~$L$ (relative to a monomial order), if it is not in~$\mathrm{LT}(L)$, the ideal of leading terms in $L$. Let $\mathcal B$ be the set of standard monomials of $L$ 
Then, $\mathcal B$ is finite if and only if $\mathbf{V}(L)$ is finite, and $\#\mathcal B$ equals the number of points in $\mathbf{V}(L)$ counting multiplicities (see \cite[Proposition 2.1]{Sturmfels2002}). 

Now, let $G=\{g_1,\ldots,g_s\}$ be a Gr\"obner basis of $I:J^\infty$ relative to the lex order from~(\ref{lex}). By Lemma \ref{GB_lemma}, there is a proper algebraic subvariety $\Delta\subsetneq \mathbb C^k$ such that $\phi_{\bfq}(G)=\{\phi_{\bfq}(g_1),\ldots,\phi_{\bfq}(g_s)\}$ is a Gr\"obner basis for~$I_\bfq:J_\bfq^\infty$ for all~$\bfq\not\in\Delta$. Moreover, none of the leading terms of $g_1,\ldots,g_s$ vanish when evaluated at $\bfq\not\in\Delta$. This implies that the leading monomials of~$I_\bfq:J_\bfq^\infty$ are constant on $\mathbb C^k\setminus \Delta$. Thus, if $\mathcal B_{\bfq}$ denotes the set of standard monomials of~$I_\bfq:J_\bfq^\infty$, also $\mathcal B_{\bfq}$ is constant on $\mathbb C^k\setminus \Delta$. On the other hand, 
$N(\bfq) = \#\mathcal B_{\bfq}$ by~(\ref{key_idea}) and the fact that regular zeros have multiplicity one.
This shows that $N(\bfq)$ is constant on~$\mathbb C^k\setminus\Delta$. The Implicit Function Theorem implies that for all $\bfq\in \mathbb C^k$ there exists a Euclidean neighbourhood $U$ of~$\bfq$ such that  $N(\bfq)\leq N(\bfq')$ for all~$\bfq'\in U$. Since~$\Delta$ is a proper subvariety of $\mathbb C^k$ and thus lower-dimensional, we have~$N(\bfq) = N<\infty$ for~$\bfq\in \mathbb C^k\setminus \Delta$. 
\end{remark}

The description of the discriminant $\Delta$ in (\ref{disc}) leads to an algorithm for computing it: given $I = \langle f_1, \ldots, f_n \rangle$, we first compute the Jacobian determinant $h=\det\big(\tfrac{\partial f_i}{\partial x_j}\big)$. Then, we compute a lex Gröbner basis for $I + \langle 1- y\cdot h \rangle$. The product of the leading coefficients~$c_i(\bfp)$ of this Gröbner basis gives us an equation for the discriminant. It is important to emphasize that this algorithm does \emph{not} yield the smallest hypersurface $\Delta$ with the properties in Theorem \ref{theorem_number_zeros}. We will see this in the next example.

~
\begin{example}
We consider again the two examples from the introduction.
\begin{enumerate}
\item We compute the discriminant for $\mathcal F = \{ax^2+bx+c\mid a,b,c\in\mathbb C\}$:

\begin{lstlisting}[language=Macaulay2]
R = QQ[y, x, a, b, c, MonomialOrder => Lex];
f = a * x^2 + b * x + c; h = 2 * a * x + b;
I = ideal {f}; 
K = ideal {1 - y * h};
G = gens gb (I+K)
\end{lstlisting}
This gives us the Gr\"obner basis 
$$\overline{G}=\{ax^2+bx+c,\ (4ac - b^2)y  + 2xa + b, yxb + 2yc + x, 2yxa + yb - 1\}.$$ The leading terms are $c_1(\bfp)=a$, $c_2(\bfp)=4ac - b^2$, $c_3(\bfp)=b$ and  $c_4(\bfp)=2a$. We get the discriminant
$$\Delta = \{\bfq = (a,b,c)\in\mathbb C^3\mid ab(4ac - b^2)=0\}.$$
Indeed, $ax^2+bx+c$ has less than two regular zeros if and only if $a=0$ or $4ac - b^2 = 0$. The additional factor $b$ is no  contradiction: we show that if~$\bfq\not\in\Delta$ then $N(\bfq)=N$ is maximal, but we do not show that $N(\bfq)=N$ implies $\bfq\not\in \Delta$.

\item In the second example, we compute the discriminant analogously:
\begin{lstlisting}[language=Macaulay2]
R = QQ[y, x1, x2, a, MonomialOrder => Lex];
gamma = x1^2 + x2^2 - 1;
f1 = (x1 - a) * (x1 - 1) * gamma;
f2 = (x2 - 3) * (x2 - 4)^2 * gamma;
h = diff(x1, f1) * diff(x2, f2) - diff(x1, f2) * diff(x2, f1);
I = ideal {f1, f2};
K = ideal {1 - y * h};
G = gens gb (I+K)
\end{lstlisting}
The Gr\"obner basis we get is $\overline{G} = \{g_1,g_2,g_3,g_4\}$ with polynomials $g_1=x_2-3$, $g_2 = x_1^2-ax_1-x_1+a$ and 
\begin{align*}
g_3 &= 81(a^6-2a^5+17a^4-32a^3+80a^2-128a+64)y\\
&\quad +(-a^4-16a^2-145)x_1+a^5+16a^3+64a+81\\
g_4 &= 13122yx_1+81(a^5-a^4+16a^3-16a^2-98a-64)y\\
&\quad +(-a^3-a^2-17a-17)x_1+a^4+a^3+17a^2+17a-81
\end{align*}
The following code then finds the discriminant:
\begin{lstlisting}[language=Macaulay2]
E = (entries(G))#0
P = QQ[a][y, x1, x2, MonomialOrder => Lex]
result = apply(E, t -> leadCoefficient(sub(t, P)))
factor(product result)
\end{lstlisting}
The result is
$$\Delta = \{a\in\mathbb C\mid (a^2+8)(a-1)=0\}$$
as the leading term of $g_3$ is $81(a^2+8)^2(a-1)^2$.
Indeed, the parameters for which we obtain less regular zeros than 2 are $a=1$ (in this case, $(a,3)$ is a double root) and $a=\pm\sqrt{-8}$ (in this case, $(a,3)$ lies on the circle $\gamma(\bfx)=0$). 
\end{enumerate}
\end{example}

\medskip
\bibliographystyle{plain}
\bibliography{literature}

\begin{thebibliography}{10}

\bibitem{NNLA}
D.~J. Bates, P.~Breiding, T.~Chen, J.~D. Hauenstein, A.~Leykin, and F.~Sottile.
\newblock Numerical {N}onlinear {A}lgebra.
\newblock {\em arXiv:2302.08585}, 2023.

\bibitem{BCSS}
L.~Blum, F.~Cucker, M.~Shub, and S.~Smale.
\newblock Complexity and {R}eal {C}omputation: {A} {M}anifesto.
\newblock {\em International Journal of Bifurcation and Chaos}, 06(01):3--26,
  1996.

\bibitem{CLO15}
D.~Cox, J.~Little, and D.~O'Shea.
\newblock {\em Ideals, {V}arieties, and {A}lgorithms: An introduction to
  computational algebraic geometry and commutative algebra}.
\newblock Undergraduate Texts in Mathematics. Springer, Cham, fourth edition,
  2015.

\bibitem{Davidenko}
D.~F. Davidenko.
\newblock On a new method of numerical solution of systems of nonlinear
  equations.
\newblock {\em Doklady Akad. Nauk SSSR (N.S.)}, 88:601--602, 1953.

\bibitem{Davidenko_full}
D.~F. Davidenko.
\newblock On approximate solution of systems of nonlinear equations.
\newblock {\em Ukrain. Mat. \v{Z}urnal}, 5:196--206, 1953.

\bibitem{10.1145/3653002.3653003}
M.~L. Dogan, A.~A. Erg\"{u}r, and E.~Tsigaridas.
\newblock On the {C}omplexity of {C}how and {H}urwitz forms.
\newblock {\em ACM Commun. Comput. Algebra}, 57(4):167–199, mar 2024.

\bibitem{Gelfand1994DiscriminantsRA}
I.~M. Gel'fand, M.~Kapranov, and A.~V. Zelevinsky.
\newblock Discriminants, {R}esultants, and {M}ultidimensional {D}eterminants.
\newblock MA: Birkhäuser Boston, 1994.

\bibitem{M2}
D.~R. Grayson and M.~E. Stillman.
\newblock Macaulay2, a software system for research in algebraic geometry.
\newblock Available at \url{http://www.math.uiuc.edu/Macaulay2/}, 2020.

\bibitem{KAPUR2013124}
D.~Kapur, Y.~Sun, and D.~Wang.
\newblock An efficient method for computing comprehensive {G}r{\"o}bner bases.
\newblock {\em Journal of Symbolic Computation}, 52:124--142, 2013.
\newblock International Symposium on Symbolic and Algebraic Computation.

\bibitem{montes1999basic}
A.~Montes.
\newblock Basic algorithms for specialization in {G}r{\"o}bner bases.
\newblock {\em ACM SIGSAM Bulletin}, 33(3):18, 1999.

\bibitem{montes2002new}
A.~Montes.
\newblock A new algorithm for discussing {G}r{\"o}bner bases with parameters.
\newblock {\em Journal of Symbolic Computation}, 33(2):183--208, 2002.

\bibitem{MS1989}
A.~P. Morgan and A.~J. Sommese.
\newblock Coefficient-parameter polynomial continuation.
\newblock {\em Applied Mathematics and Computation}, 29(2):123--160, 1989.

\bibitem{Rojas2002WhyPM}
J.~M. Rojas.
\newblock Why polyhedra matter in non-linear equation solving.
\newblock In {\em Topics in algebraic geometry and geometric modeling}, volume
  334 of {\em Contemp. Math.}, pages 293--320. Amer. Math. Soc., Providence,
  RI, 2003.

\bibitem{FTA}
J.~M. Rojas.
\newblock On the {BCSS} proof of the {F}undamental {T}heorem of {A}lgebra.
\newblock {\em arXiv:2406.12198}, 2024.

\bibitem{Sommese:Wampler:2005}
A.~J. Sommese and C.~W. Wampler.
\newblock {\em The {N}umerical {S}olution of {S}ystems of {P}olynomials
  {A}rising in {E}ngineering and {S}cience}.
\newblock World Scientific, 2005.

\bibitem{Sturmfels2002}
B.~Sturmfels.
\newblock {\em {Solving Systems of {P}olynomial {E}quations}}.
\newblock Number~97 in CBMS Regional Conferences Series. American Mathematical
  Society, 2002.

\bibitem{STURMFELS2017186}
B.~Sturmfels.
\newblock The {H}urwitz form of a projective variety.
\newblock {\em Journal of Symbolic Computation}, 79:186--196, 2017.
\newblock SI: MEGA 2015.

\bibitem{WEISPFENNING19921}
V.~Weispfenning.
\newblock Comprehensive {G}r{\"o}bner bases.
\newblock {\em Journal of Symbolic Computation}, 14(1):1--29, 1992.

\bibitem{WEISPFENNING2003669}
V.~Weispfenning.
\newblock Canonical comprehensive {G}r{\"o}bner bases.
\newblock {\em Journal of Symbolic Computation}, 36(3):669--683, 2003.
\newblock ISSAC 2002.

\end{thebibliography}

\bigskip

\end{document}